\documentclass[reqno]{amsart}
\usepackage{amsmath, amsthm, amssymb, mathabx}
\begin{document}
\newtheorem{theo}{Theorem}[section]
\newtheorem*{thm*}{Theorem A}
\newtheorem{defin}[theo]{Definition}
\newtheorem{rem}[theo]{Remark}
\newtheorem{lem}[theo]{Lemma}
\newtheorem{cor}[theo]{Corollary}
\newtheorem{prop}[theo]{Proposition}
\newtheorem{exa}[theo]{Example}
\newtheorem{exas}[theo]{Examples}
\numberwithin{equation}{section}
%
%
\thanks{}
\title[High energy sign-changing solutions]{High energy sign-changing solutions to Schr\"{o}dinger-Poisson type systems}

\author[]{Cyril Joel Batkam}
\address{
The Fields Institute,
\newline
Toronto, Ontario, M5T 3J1, CANADA.}
\email{cbatkam@fields.utoronto.ca}
\maketitle
\begin{abstract}
We prove the existence of infinitely many high energy sign-changing solutions for some classes of Schr\"{o}dinger-Poisson systems in bounded domains, with nonlinearities having subcritical or critical growth. Our approach is variational and relies on an application of a new sign-changing version of the symmetric mountain pass theorem.
\end{abstract}
%
\section{Introduction}
This paper is concerned with the existence of high energy sign-changing solutions to some elliptic systems of Schr\"{o}dinger-Poisson type. The Schr\"{o}dinger-Poisson system is a simple model used for the study of quantum transport in semiconductor devices that can be written as
\begin{equation}\label{sp1}
   \left\{
      \begin{array}{ll}
       -\imath\hbar\frac{\partial \psi}{\partial t}(t,x)= -\frac{\hbar^2}{2m}\Delta \psi(t,x)+\phi(t,x)\psi(t,x)-g\big(\psi(t,x)\big)\quad \text{in }\Omega & \hbox{} \\
       \,\,\,\, -\Delta \phi(t,x)=|\psi(t,x)|^2\qquad\qquad\qquad \qquad\qquad\quad\quad\qquad\quad\,\,\text{in }\Omega, & \hbox{} 
      \end{array}
    \right.
\end{equation}
where $\Omega$ is a domain in $\mathbb{R}^3$, $\hbar$ is the Planck constant, $m$ is the particle mass, $\Delta$ is the usual spatial Laplace operator, $\psi:\mathbb{R}\times\overline{\Omega}\to\mathbb{C}$ is the wave function, $\phi:\mathbb{R}\times\overline{\Omega}\to\mathbb{C}$ is the electrostatic potential, and $g:\mathbb{R}\to\mathbb{C}$ is a nonlinear term describing external perturbations and interactions between many particles. In case $\Omega$ is bounded, one can impose the boundary conditions $\psi=\phi=0$, which physically mean that the particles are constraint to live in $\Omega$. For more information on the physical relevance of the Schr\"{o}dinger-Poisson system, we refer to \cite{BF98,M90,N90}.
\par The study of standing wave solutions of \eqref{sp1}, that is solutions of the form 
\begin{equation*}
\Big(\psi(t,x)=e^{-\imath \omega t}u(x),\phi(t,x)\Big),\quad u(x)\in\mathbb{R},\,\,\omega>0, 
\end{equation*}
 leads to a problem of the form
\begin{equation}\label{sp2}
 \qquad   \left\{
      \begin{array}{ll}
        -\Delta u+\phi u=h(x,u)\quad \text{in }\Omega & \hbox{} \\
       \qquad\,\, -\Delta \phi=u^2\qquad \quad\text{in }\Omega, & \hbox{} 
       \end{array}
    \right.
\end{equation}
which has been widely study in the last decade. Many important results concerning existence and non existence of solutions, multiplicity of solutions, least energy solutions, radial and non radial solutions, semiclassical limit and concentrations of solution have been obtained. See for instance \cite{AR08,AP08,BF98,C03,CV10,IV08,RS08,S10,WZ07,ZZ08} and the references quoted there. In these papers, some of the solutions found are nonnegative, but in many cases the sign of the solutions cannot be decided. The existence of a sign-changing solutions to \eqref{sp2} was considered recently by Alves and Souto \cite{AS14}, Kim and Seok \cite{KS12} and Wang and Zhou \cite{WZ}. They obtained a non trivial sign-changing solution by using the method of the Nehari manifold. We recall that a solution $(u,\phi)$ to \eqref{sp2} is said to be sign-changing if $u$ changes its sign. The existence of many sign-changing solutions to \eqref{sp2} was an open problem until the recent work of Liu, Wang and Zhang \cite{LWZ}. In that very nice paper, they considered \eqref{sp2} in the whole space $\mathbb{R}^3$ and they obtained infinitely many sign-changing solutions by using a new version of the symmetric mountain pass theorem in the presence of invariant sets of the descending flow established in \cite{LLW}. However, the nonlinear term $h$ was assumed to be of subcritical growth. Moreover, it seems that there is no result in literature on the existence of sign-changing solutions to the Schr\"{o}dinger-Poisson system with critical growing nonlinearities.
\par The first goal of this paper is to prove the existence of (many) sign-changing solutions to \eqref{sp2} by allowing the nonlinear function $h$ to contain a critical term. More precisely, we investigate the existence of multiple sign-changing solutions to the system
\begin{equation*}
 (SP)_\lambda\qquad   \left\{
      \begin{array}{ll}
        -\Delta u+\phi u=f(x,u)+\lambda u^5\quad \text{in }\Omega & \hbox{} \\
         \qquad\,\, -\Delta \phi=u^2\qquad\qquad\qquad \text{in }\Omega & \hbox{} \\
        \qquad u=\phi=0 \qquad \qquad\qquad  \text{ on }\partial\Omega, & \hbox{}
      \end{array}
    \right.
\end{equation*}
where $\Omega$ is an open bounded subset of $\mathbb{R}^3$ with smooth boundary, $\lambda\geq0$, and $f:\overline{\Omega}\to\mathbb{R}$ satisfies the following conditions:
\vspace{0.3cm}
\begin{enumerate}
  \item[$(f_1)$] $f\in C\big(\overline{\Omega}\times\mathbb{R},\mathbb{R}\big)$ and there exists a constant $c>0$ such that
\begin{equation*}
    |f(x,u)|\leq c\big(1+|u|^{p-1}\big),\quad\text{where }4<p<6;
\end{equation*}
  \item[$(f_2)$] $f(x,u)=\circ(|u|)$, uniformly in $x\in \overline{\Omega}$, as $u\to0$;
  \vspace{0.3cm}
  \item[$(f_3)$] there exists $\mu>4$ such that $0<\mu F(x,u)\leq uf(x,u)$ for all $u\neq0$ and for all $x\in\overline{\Omega}$, where $F(x,u)=\int_0^u f(x,s)ds$;
  \vspace{0.3cm}
  \item [$(f_4)$] $f(x,-u)=-f(x,u)$ for all $(x,u)\in\overline{\Omega}\times\mathbb{R}$.
\end{enumerate}
\vspace{0.3cm}
 We shall prove the following results.
\begin{theo}[Subcritical case]\label{result1}
Assume that $\lambda=0$. If $(f_{1,2,3,4})$ are satisfied then $(SP)_0$ has a sequence $\big(u_k,\phi_k\big)_{k\geq1}$ of solutions such that $u_k$ is sign-changing and
\begin{equation*}
\frac{1}{2}\int_\Omega|\nabla u_k|^2+\frac{1}{4}\int_\Omega\phi_ku_k^2-\int_\Omega F(x,u_k)\to\infty,\quad\text{as }k\to\infty.
\end{equation*}
\end{theo}
\begin{theo}[Critical case]\label{result2}
Assume that $(f_{1,2,3,4})$ are satisfied. Then there exists a sequence $\big(\lambda_k,u_k,\phi_k\big)_{k\geq1}$ such that $\lambda_k\to0^+$, $(u_k,\phi_k)$ is solution to $(SP)_{\lambda_k}$, $u_k$ is sign-changing, and
\begin{equation*}
\frac{1}{2}\int_\Omega|\nabla u_k|^2+\frac{1}{4}\int_\Omega\phi_ku_k^2dx-\int_\Omega F(x,u_k)-\frac{\lambda_k}{6}\int_\Omega u_k^6\to\infty,\quad\text{as }k\to\infty.
\end{equation*}
\end{theo}
\begin{rem}
It was claimed by the authors of \cite{LWZ} that Theorem \ref{result1} holds \big(see Remark 1.1 in \cite{LWZ}\big). However, the proof we provide here, which is based on a different approach, appears to be much more simpler.
\end{rem}
Our approach in proving Theorems \ref{result1} and \ref{result2} is variational and relies on a new sign-changing critical point theorem, we established in a recent paper \cite{B}, which is modelled on the fountain theorem of Barstch \cite{B93}.
\par In the second part of this paper, we will use the same approach to study the following Schr\"{o}dinger-Poisson type system:
\begin{equation*}
 (SP1)\qquad   \left\{
      \begin{array}{ll}
        -\Delta u+\psi u^3=|u|^{q-2}u\quad \text{in }\Omega & \hbox{} \\
        \quad\quad\,\,\,\,\,-\Delta \psi=\frac{1}{2} u^4\qquad\, \text{ in }\Omega & \hbox{} \\
       \quad\quad\,\,\, u=\psi=0 \qquad\quad\,\,\,\text{on }\partial\Omega, & \hbox{}
      \end{array}
    \right.
\end{equation*}
where $\Omega$ is a bounded domain in $\mathbb{R}^N$ with smooth boundary. This problem was first introduced by Azzollini, d'Avenia and Luisi in \cite{ADaL13}. By combining the method of cut-off function with variational arguments, they proved that $(SP1)$ possesses at least one non trivial solution when $\Omega$ is contained in $\mathbb{R}^3$ and $1<q<5$. In  \cite{ADa12,LLS14,LZ13}, the authors studied the case where the non local term grows critically. However, in those papers only nonnegative solutions were found. As far as we know, there is no result concerning the existence of sign-changing solutions of $(SP1)$. Therefore, the second goal of this paper is to prove that $(SP1)$ possesses sign-changing solutions.
\par Our result on this problem reads as follows:
\begin{theo}\label{result3}
Let $\Omega$ be a bounded smooth domain in $\mathbb{R}^N$ ($N=1,2$). If $q>8$ then $(SP1)$ has a sequence $\big(u_k,\psi_k\big)_{k\geq1}$ of solutions such that $u_k$ is sign-changing and
\begin{equation*}
\frac{1}{2}\int_\Omega|\nabla u_k|^2+\frac{1}{8}\int_\Omega\psi_ku_k^4-\frac{1}{q}\int_\Omega |u_k|^q\to\infty,\quad\text{as }k\to\infty.
\end{equation*}
\end{theo}
\begin{rem}
As a consequence of the Lax-Milgram theorem, $(SP1)$ can be transformed into a single semilinear elliptic equation such that the non local term is homogeneous of degree 8 (see Section \ref{quatre}). Therefore, condition $q>8$ in the statement of Theorem \ref{result3} is needed to guarantee that the problem possesses the mountain pass structure, which is crucial for our argument.
However, it seems that this condition can be relaxed by means of the perturbation method used in \cite{LWZ}. 
\end{rem}
The paper is organized as follows. In Section \ref{deux}, we state the abstract critical point theorem we will apply in the proof of our main results. In Section \ref{trois}, we provide the proof of Theorems \ref{result1} and \ref{result2}. Finally, we prove Theorem \ref{result3} in Section \ref{quatre}.
\par Throughout the paper we denote by $|\cdot|_q$ the norm of the Lebesgue space $L^p(\Omega)$, by $"\rightarrow"$ the strong convergence and by $"\rightharpoonup"$ the weak convergence.
\section{Abstract preliminary}\label{deux}
In this section, we present the critical point theorem which will be applied to prove our main results. 
\par Let $\Phi$ be a $C^1$-functional defined on a Hilbert space $X$ of the form
\begin{equation}\label{space}
X:=\overline{\oplus_{j=0}^\infty X_j},\quad\text{with } \dim X_j<\infty.
\end{equation}
We introduce for $k\geq2$ and $m>k+2$ the following notations:
\begin{equation*}
Y_k:=\oplus_{j=0}^k X_j,\quad Z_k=\overline{\oplus_{j=k}^\infty X_j},\quad Z^m_k=\oplus_{j=k}^m X_j,\quad  B_k:=\big\{u\in Y_k\,|\, \|u\|\leq\rho_k\big\}, 
\end{equation*}
\begin{equation*}
N_k:=\big\{u\in Z_k\,|\,\|u\|=r_k\big\},\,\, N^m_k:=\big\{u\in Z^m_k\,|\,\|u\|=r_k\big\},\,\,\text{where }0<r_k<\rho_k,
\end{equation*}
\begin{equation*}
\Phi_m:=\Phi|_{Y_m},\, K_m:=\big\{u\in Y_m\,;\, \Phi'_m(u)=0\big\}\text{ and }\,E_m:=Y_m\backslash K_m.
\end{equation*}
Let $P_m$ be a closed convex cone of $Y_m$. We set for $\mu_m>0$
\begin{equation*}
\pm D_m^0:=\big\{u\in Y_m\,|\, dist\big(u,\pm P_m\big)<\mu_m\big\},\,\,  D_m=D_m^0\cup(-D_m^0)\,\text{and } S_m:=Y_m\backslash D_m.
\end{equation*}
We will also denote the $\alpha$-neighborhood of $S\subset Y_m$ by
\begin{equation*}
V_\alpha(S):=\big\{u\in Y_m\,|\, dist(u,S)\leq\alpha\big\},\quad\forall\alpha>0.
\end{equation*}
The following result was established by the present author in \cite{B}.  For the sake of completeness, the proof will be provided in this paper.
\begin{theo}\label{scft}
Let $\Phi\in C^1(X,\mathbb{R})$ be an even functional which maps bounded sets to bounded sets. If, for $k\geq2$ and $m>k+2$, there exist $0<r_k<\rho_k$ and $\mu_m>0$ such that
\begin{enumerate}
\item[$(A_1)$] $a_k:=\max_{\substack{u\in \partial B_k}}\Phi(u)<b_k:=\inf_{\substack{u\in N_k}}\Phi(u)$.
\item[$(A_2)$] $N^m_k\subset S_m$.
\item[$(A_3)$] There exists an odd locally Lipschitz continuous vector field $B:E_m\to Y_m$ such that:
\begin{itemize}
\item[(i)] $B\big((\pm D_m^0)\cap E_m\big)\subset \pm D_m^0$;
\item[(ii)] there exists a constant $\alpha_1>0$ such that $\big<\Phi'_m(u),u-B(u)\big>\geq\alpha_1\|u-B(u)\|^2$, for any $u\in E_m$;
\item[(iii)] for $a<b$ and $\alpha>0$, there exists $\beta>0$ such that $\|u-B(u)\|\geq\beta$ if $u\in Y_m$ is such that $\Phi_m(u)\in[a,b]$ and $\|\Phi'_m(u)\|\geq\alpha$.
\end{itemize}
\end{enumerate}
Then there exists a sequence $(u_{k,m}^n)_n\subset V_{\frac{\mu_m}{2}}(S_m)$ such that
\begin{equation*}
\lim_{\substack{n\to\infty}}\Phi'_m(u_{k,m}^n)=0\quad\text{and}\quad \lim_{\substack{n\to\infty}}\Phi_m(u_{k,m}^n)\in\big[b_k,\max_{\substack{u\in B_k}}\Phi(u)\big].
\end{equation*}
\end{theo}
 The proof of this theorem relies on the following deformation lemma.
\begin{lem}\label{defor}
Let $\Phi\in C^1(X,\mathbb{R})$ be an even functional which maps bounded sets to bounded sets. Fix $m$ sufficiently large and assume that the condition $(A_3)$ of Theorem \ref{scft} holds. Let $c\in\mathbb{R}$ and $\varepsilon_0>0$ such that
\begin{equation}\label{one}
\forall u\in\Phi_m^{-1}\big([c-2\varepsilon_0,c+2\varepsilon_0]\big)\cap V_{\frac{\mu_m}{2}}(S_m)\,:\, \|\Phi_m'(u)\|\geq\varepsilon_0.
\end{equation}
Then for some $\varepsilon\in]0,\varepsilon_0[$ there exists $\eta\in C\big([0,1]\times Y_m,Y_m\big)$ such that:
\begin{enumerate}
\item[(i)] $\eta(t,u)=u$ for $t=0$ or $u\notin \Phi_m^{-1}\big([c-2\varepsilon,c+2\varepsilon]\big)$;
\item[(ii)] $\eta\big(1,\Phi_m^{-1}(]-\infty,c+\varepsilon])\cap S_m\big)\subset \Phi_m^{-1}\big(]-\infty,c-\varepsilon]\big)$;
\item[(iii)] $\Phi_m\big(\eta(\cdot,u)\big)$ is not increasing, for any $u$;
\item[(iv)] $\eta([0,1]\times D_m)\subset D_m$;
\item[(v)] $\eta(t,\cdot)$ is odd, for any $t\in[0,1]$.
\end{enumerate} 
\end{lem}
The proof of this lemma is given in the appendix.
\begin{proof}[Proof of Theorem \ref{scft}]
$(A_1)$ and $(A_2)$ imply that for $k$ big enough we have $a_k<b_k\leq\inf_{\substack{u\in N_k^m}}\Phi_m(u)$. Let
\begin{equation*}
\Gamma_k^m:=\big\{\gamma\in C(B_k,Y_m)\,;\, \gamma \text{ is odd, } \gamma|_{\partial B_k}=id\text{ and }\gamma(D_m)\subset D_m\big\}.
\end{equation*}
$\Gamma_k^m$ is clearly non empty and for any $\gamma\in \Gamma_k^m$ the set $U:=\big\{u\in B_k\,;\,\|\gamma(u)\|<r_k\big\}$ is an open bounded and symmetric \big(i.e. $-U=U$\big) neighborhood of the origin in $Y_k$. By the Borsuk-Ulam theorem, the continuous odd map $\Pi_k\circ\gamma:\partial U\subset Y_k\to Y_{k-1}$ has a zero, where $\Pi_k:X\to Y_{k-1}$ is the orthogonal projection. It then follows that $\gamma(B_k)\cap N_k^m\neq\emptyset$ and, since $N_k^m\subset S_m$, that $\gamma(B_k)\cap S_m\neq\emptyset$. This intersection property implies that
\begin{equation*}
c_{k,m}:=\inf_{\gamma\in\Gamma_k^m}\max_{u\in\gamma(B_k)\cap S_m}\Phi_m(u)\geq \inf_{\substack{u\in N^m_k}}\Phi(u)\geq b_k.
\end{equation*}
We would like to show that for any $\varepsilon_0\in]0,\frac{c_{k,m}-a_k}{2}[$, there exists $u\in\Phi_m^{-1}\big([c_{k,m}-2\varepsilon_0,c_{k,m}+2\varepsilon_0]\big)\cap V_{\frac{\mu_m}{2}}(S_m)$ such that $\|\Phi'_m(u)\|<\varepsilon_0$.\\
Arguing by contradiction, we assume that we can find $\varepsilon_0\in]0,\frac{c_{k,m}-a_k}{2}[$ such that 
\begin{equation*}
\|\Phi'_m(u)\|\geq\varepsilon_0,\quad\forall u\in\Phi_m^{-1}\big([c_{k,m}-2\varepsilon_0,c_{k,m}+2\varepsilon_0]\big)\cap V_{\frac{\mu_m}{2}}(S_m).
\end{equation*}
Apply Lemma \ref{defor} with $c=c_{k,m}$ and define, using the deformation $\eta$ obtained the map
\begin{equation*}
\theta:B_k\to Y_m,\qquad \theta(u):=\eta(1,\gamma(u)),
\end{equation*}
where $\gamma\in\Gamma_k^m$ satisfies
\begin{equation}\label{six}
\max_{\substack{u\in\gamma(B_k)\cap S_m}}\Phi_m(u)\leq c_{k,m}+\varepsilon,
\end{equation}
with $\varepsilon$ also given by Lemma \ref{defor}.\\
Using the properties of $\eta$ (see Lemma \ref{defor}), one can easily verify that $\theta\in\Gamma^m_k$.\\
On the other hand, we have
\begin{equation}\label{five}
\eta\big(1,\gamma(B_k)\big)\cap S_m\subset \eta\big(1,\Phi_m^{-1}(]-\infty,c_{k,m}+\varepsilon])\cap S_m\big).
\end{equation}
In fact, if $u\in \eta\big(1,\gamma(B_k)\big)\cap S_m$ then $u=\eta\big(1,\gamma(v)\big)\in S_m$ for some $v\in B_k$. Observe that $\gamma(v)\in S_m$. Indeed, if this is not true then we have $\gamma(v)\in D_m$, and by (iv) of Lemma \ref{defor} we have $u=\eta(1,\gamma(v))\in D_m$ which contradicts the fact that $u\in S_m$. Since by \eqref{six} $\gamma(v)\in\Phi_m^{-1}(]-\infty,c_{k,m}+\varepsilon])$, we deduce using (ii) of Lemma \ref{defor} that $u=\eta(1,\gamma(v))\in\eta\big(1,]-\infty,c_{k,m}+\varepsilon]\cap S_m\big)$. Hence \eqref{five} holds.\\
Using \eqref{five} and (ii) of Lemma \ref{defor}, we obtain
\begin{align*}
\max_{\substack{u\in\theta(B_k)\cap S}}\Phi_m(u)&=\max_{\substack{u\in\eta\big(1,\gamma(B_k)\big)\cap S_m}}\Phi_m(u)\\
&\leq \max_{\substack{u\in\eta\big(1,\Phi_m^{-1}(]-\infty,c_{k,m}+\varepsilon])\cap S_m\big)}}\Phi_m(u)\\
&\leq c_{k,m}-\varepsilon,
\end{align*}
contradicting the definition of $c_{k,m}$.\\
The above contradiction assures that for any $\varepsilon_0\in]0,\frac{c_{k,m}-a_k}{2}[$, there exists $u\in\Phi_m^{-1}\big([c_{k,m}-2\varepsilon_0,c_{k,m}+2\varepsilon_0]\big)\cap V_{\frac{\mu_m}{2}}(S_m)$ such that $\|\Phi'_m(u)\|<\varepsilon_0$.\\
We then deduce by letting $\varepsilon_0$ goes to $0$ that there is a sequence $(u^n_{k,m})_n\subset V_{\frac{\mu_m}{2}}(S_m)$ such that
\begin{equation*}
\quad \Phi'_m(u^n_{k,m})\to0\text{ and }\Phi_m(u^n_{k,m})\to c_{k,m},\text{ as }n\to\infty.
\end{equation*}
\end{proof}
\section{Proof of Theorems \ref{result1} and \ref{result2}}\label{trois}
In this section, we treat the system
\begin{equation*}
 (SP)_\lambda\qquad   \left\{
      \begin{array}{ll}
        -\Delta u+\phi u=f(x,u)+\lambda u^5\quad \text{in }\Omega & \hbox{} \\
       \qquad\,\, -\Delta \phi=u^2\qquad\qquad\qquad \text{in }\Omega & \hbox{} \\
        \qquad u=\phi=0 \qquad \qquad\qquad  \text{ on }\partial\Omega, & \hbox{}
      \end{array}
    \right.
\end{equation*}
We will assume throughout this section that $(f_{1,2,3,4})$ are satisfied.  
\par Let $X:=H_0^1(\Omega)$ be the usual Sobolev space endowed with the inner product
\begin{equation*}
\big<u,v\big>=\int_\Omega\nabla u\nabla v
\end{equation*}
and norm $\|u\|^2=\big<u,u\big>$, for $u,v\in H_0^1(\Omega)$.
\par For any fixed $u\in H_0^1(\Omega)$, the Lax-Milgram theorem implies that the problem
\begin{equation*}
-\Delta \phi=u^2,\quad \phi\in H_0^1(\Omega) 
\end{equation*}
has a unique solution $\phi_u$. Moreover, $\phi_u$ has the following properties \big(see e.g \cite{R06} for a proof\big):
\begin{prop}
For $u\in H_0^1(\Omega)$ we have
\begin{enumerate}
\item[(i)] $\phi_u\geq0$ and there exists $C>0$ such that $\|\phi_u\|\leq C\|u\|^2$;
\item[(ii)] $\phi_{tu}=t^2\phi_u$, for all $t\geq0$;
\item[(iii)] If $u_n\rightharpoonup u$ in $H_0^1(\Omega)$, then $\phi_{u_n}\rightharpoonup\phi_u$ in $H_0^1(\Omega)$ and
\begin{equation}
\int_\Omega\phi_{u_n}u_n^2\to\int_\Omega\phi_uu^2.
\end{equation}
\end{enumerate}
\end{prop} 
As a consequence of this proposition, $(u,\phi)\in H_0^1(\Omega)\times H_0^1(\Omega)$ is a solution to $(SP)_\lambda$ if, and only if $\phi=\phi_u$ and $u$ is solution to the non local problem
\begin{equation}\label{e1}
-\Delta u+\phi_uu=f(x,u)+\lambda u^5,\quad u\in H_0^1(\Omega).
\end{equation}
Problem \eqref{e1} is variational and its solutions are critical points of the functional defined in $H_0^1(\Omega)$ by 
\begin{equation}\label{e2}
I_\lambda(u):=\frac{1}{2}\|u\|^2+\frac{1}{4}\int_\Omega\phi_uu^2-\int_\Omega F(x,u)-\frac{\lambda}{6}\int_\Omega u^6.
\end{equation}
By using standard argument one can verify that $I_\lambda\in C^1(X,\mathbb{R})$ and 
\begin{equation}\label{e3}
\big<I_\lambda'(u),v\big>=\int_\Omega \nabla u\nabla v+\int_\Omega \phi_uuv-\int_\Omega vf(x,u)-\lambda\int_\Omega vu^5.
\end{equation}
\par Let $0<\sigma_1<\sigma_2<\sigma_3<\cdots$ be the distinct eigenvalues of the Laplacian. Then each $\sigma_j$ has finite multiplicity. It is well known that the principal eigenvalue $\sigma_1$ is simple with a positive eigenfunction $e_1$, and the eigenfunctions $e_j$ corresponding to $\sigma_j$ ($j\geq2$) are sign-changing. Let $X_j$ be the eigenspace of $\sigma_j$. \\
We set for $k\geq2$
\begin{equation}\label{ykzk}
Y_k:=\oplus_{j=1}^k X_j\text{ and } Z_k=\overline{\oplus_{j=k}^\infty X_j}. 
\end{equation}
\begin{lem}\label{l1}\quad
\begin{enumerate}
\item[(1)] For any $u\in Y_k$, we have $I_\lambda(u)\to-\infty$, uniformly in $\lambda$ as $\|u\|\to\infty$.
\item[(2)] There exist $\Lambda_k^1>0$, $r_k>0$ with $r_k\to\infty$ as $k\to\infty$, such that
\begin{equation}\label{e4}
I_\lambda(u)\geq\frac{1}{8}r_k^2-c
\end{equation}
\end{enumerate}
for all $\lambda\in[0,\Lambda_k^1[$ and for all $u\in Z_k$ such that $\|u\|=r_k$, where $c>0$ is constant.
\end{lem}
\begin{proof}
(1)\, We recall that $(f_3)$ implies that $F(x,u)\geq c_1|u|^\mu-c_2$ for some constants $c_1,c_2>0$. Since all norms are equivalent in the finite-dimensional space $Y_k$, it is easy to conclude.\\
(2)\, Using $(f_1)$ we obtain
\begin{equation*}
I_\lambda(u)\geq\frac{1}{2}\|u\|^2-c_3|u|_p^p-\frac{\lambda}{6}|u|_6^6-c_4,\quad \forall u\in X,
\end{equation*}
where $c_3,c_4>0$ are constant. If we set
\begin{equation}\label{sbetak}
S:=\inf_{\substack{u\in H_0^1(\Omega)\\ u\neq0}}\frac{|\nabla u|_2^2}{|u|_6^2}\,\text{ and }\, \beta_k:=\sup_{\substack{u\in Z_k\\ \|u\|=1}}|u|_p,
\end{equation}
then we obtain
\begin{equation*}
I_\lambda(u)\geq\frac{1}{2}\|u\|^2-c_3\beta_k^p\|u\|^p-\frac{\lambda}{6S^3}\|u\|^6-c_4,\quad \forall u\in Z_k.
\end{equation*}
We define
\begin{equation}\label{e5}
r_k:=\big(\frac{1}{8c_3}\big)^{\frac{1}{p-2}}\beta_k^{-\frac{p}{p-2}}\,\text{ and }\, \Lambda_k^1:=\frac{3S^3}{2r_k^4}.
\end{equation}
One can easily verify that the inequality \eqref{e4} is satisfied. The fact that $r_k\to\infty$, as $k\to\infty$, is a consequence of Theorem $3.8$ in \cite{W}.
\end{proof}
\begin{rem}\label{rembk}
Lemma \ref{l1}-(2) implies that
\begin{equation}
b_k:=\inf_{\substack{u\in Z_k\\ \|u\|=r_k}}I_\lambda(u)\to\infty,\quad\text{uniformly in } \lambda,\text{ as }k\to\infty.
\end{equation}
\end{rem}
Now we fix $k$ large enough and we set for $m>k+2$
\begin{equation*}
I_{\lambda,m}:=I_\lambda|_{Y_m},\, K_m:=\big\{u\in Y_m\,;\, I_{\lambda,m}'(u)=0\big\},\, E_m:=Y_m\backslash K_m,
\end{equation*}
\begin{equation*}
P_m:=\big\{u\in Y_m\,;\,u(x)\geq0\big\},\,\, Z^m_k:=\oplus_{j=k}^m X_j\,\text{ and }  N^m_k:=\big\{u\in Z^m_k\,|\,\|u\|=r_k\big\}.
\end{equation*}
Remark that for all $u\in P_m\backslash\big\{0\big\}$ we have $\int_\Omega ue_1>0$, while for all $u\in Z_k$, $\int_\Omega ue_1=0$, where $e_1$ is the principal eigenfunction of the Laplacian. This implies that $P_m\cap Z_k=\big\{0\big\}$. Since  $N_k^m$ is compact, it follows that
\begin{equation}\label{deltam}
\delta_m:=dist\big(N_k^m,-P_m\cup P_m\big)>0.
\end{equation}
\begin{rem}\label{sm}
If $0<\mu_m<\delta_m$ then $N_k^m\subset S_m$.
\end{rem}
\par For $u\in Y_m$ fixed, we consider the functional 
\begin{equation*}
\kappa_u(v)=\frac{1}{2}\|v\|^2+\frac{1}{2}\int_\Omega\phi_uv^2-\int_\Omega vf(x,u),\quad v\in Y_m.
\end{equation*}
One can easily verify that $\kappa_u$ is continuous, coercive, bounded below, weakly sequentially continuous, and strictly convex. Therefore, $\kappa_u$ possesses a unique minimizer, namely $Au$, which is the unique solution to the problem
\begin{equation*}
-\Delta v+\phi_uv=f(x,u)+\lambda u^5,\quad v\in Y_m.
\end{equation*}
Clearly, the set of fixed points of $A$ coincide with $K_m$. Moreover, the operator $A:Y_m\to Y_m$ has the following properties.
\begin{lem}\label{Aop}\quad
\begin{enumerate}
\item[(1)] $A$ is continuous and it maps bounded sets to bounded sets.
\item[(2)] For any $u\in Y_m$ we have
\begin{align}
\big<I_{\lambda, m}'(u),u-Au\big>\geq c_1\|u-Au\|^2, \label{aa1}\\
\|I_{\lambda, m}'(u)\|\leq c_2\big(1+\|u\|^2\big)\|u-Au\|.\label{aa2}
\end{align}
\item[(3)] There exists $\mu_m\in]0,\delta_m[$ such that $A(\pm D^0_m)\subset \pm D^0_m$, where $\delta_m$ is defined by \eqref{deltam}.
\end{enumerate}
\end{lem}
The three assertions of this lemma can be proved in the same way as Lemma 3.1, Lemma 3.2 and Lemma 3.4 in \cite{LWZ},  respectively. We shall provide more details when we shall prove their analogues in Section \ref{quatre} below.
\par It should be noted that the vector field $A$ itself does not satisfy the assumption $(A_3)$ of Theorem \ref{scft} as it is not locally Liptschitz continuous. However, it is the first step in the construction of a vector field satisfying the above mentioned condition.
\begin{lem}[\cite{LWZ}, Lemma 3.5]\label{Bop}
There exists an odd locally Lipschitz continuous operator $B:E_m\to Y_m$ such that
\begin{enumerate}
\item[(1)] $\big<I_{\lambda, m}'(u),u-B(u)\big>\geq \frac{1}{2}\|u-A(u)\|^2$, for any $u\in E_m$.
\item[(2)] $\frac{1}{2}\|u-B(u)\|\leq \|u-A(u)\|\leq 2\|u-B(u)\|$, for any $u\in E_m$.
\item[(3)] $B\big((\pm D^0_m)\cap E_m\big)\subset \pm D^0_m$.
\end{enumerate} 
\end{lem}
\begin{lem}[\cite{LWZ}, Lemma 3.3]\label{tech}
Let $c<d$ and $\alpha>0$. For all $u\in Y_m$ such that $I_{\lambda, m}(u)\in[c,d]$ and $\|I_{\lambda, m}'(u)\|\geq\alpha$, there exists $\beta>0$ such that $\|u-B(u)\|\geq\beta$.
\end{lem}
\begin{lem}\label{lemam}
For any $\lambda\in[0,\Lambda_k^1[$, there exists a sequence $u_{k,m}\in V_{\frac{\mu_m}{2}}(S_m)$ such that
\begin{equation}
I'_{\lambda, m}(u_{k,m})=0\quad\text{and}\quad I_\lambda(u_{k,m})\in[b_k,\max_{u\in B_k}I_0(u)].
\end{equation}
\end{lem}
\begin{proof}
By Lemmas \ref{l1}, \ref{Bop} and \ref{tech} and Remark \ref{sm}, the assumptions of Theorem \ref{scft} are satisfied. Therefore, applying Theorem \ref{scft} we obtain a sequence $(u_{k,m}^n)_n\subset V_{\frac{\mu_m}{2}}(S_m)$ such that
\begin{equation*}
\lim_{\substack{n\to\infty}}I'_{\lambda,m}(u_{k,m}^n)=0\quad\text{and}\quad \lim_{\substack{n\to\infty}}\Phi_m(u_{k,m}^n)\in\big[b_k,\max_{\substack{u\in B_k}}I_0(u)\big].
\end{equation*}
Using $(f_1)$ and $(f_3)$ and the fact that $\phi_u\geq0$, we obtain
 \begin{equation}\label{bd}
 I_\lambda(u)-\frac{1}{\mu}\big<I'_{\lambda,m}(u),u\big>\geq \big(\frac{1}{2}-\frac{1}{\mu}\big)\|u\|^2,\quad\forall u\in Y_m.
 \end{equation}
We then deduce that the sequence $(u_{k,m}^n)_n$ above is bounded in $Y_m$. Since $\dim Y_m<\infty$, it follows that $u_{k,m}^n\to u_{k,m}$, up to a subsequence, in $Y_m$. Since $V_{\frac{\mu_m}{2}}(S_m)$ is closed and $I_{\lambda,m}$ is smooth, the conclusion follows.
\end{proof}
We are now ready to give the proof of Theorem \ref{result1}.
\begin{proof}[Proof of Theorem \ref{result1}]
Here we assume that $\lambda=0$. We consider the elements $u_{k,m}$ obtained in Lemma \ref{lemam}. In view of Lemma \ref{bd}, the sequence $(u_{k,m})_m$ is bounded in $X$. Hence, up to a subsequence, $u_{k,m}\rightharpoonup u_k$ in $X$ and $u_{k,m}\to u_k$ in $L^q(\Omega)$ ($1\leq q<6$), as $m\to\infty$.\\
Let us denote by $\Pi_m:X\to Y_m$ the orthogonal projection. It is clear that $\Pi_m u_k\to u_k$ in $X$, as $m\to\infty$. We have
 \begin{multline}
 \big<I'_{0,m}(u_{k,m}),u_{k,m}-\Pi_mu_k\big>=\big<u_{k,m},u_{k,m}-\Pi_mu_k\big>\\+\int_\Omega \phi_{u_{k,m}}u_{k,m}\big(u_{k,m}-\Pi_mu_k\big)
 -\int_\Omega \big(u_{k,m}-\Pi_mu_k\big)f(x,u_{k,m}).\label{star}
 \end{multline}
Since the sequence $\big(u_{k,m}\big)_m$ is bounded, we deduce from $(f_1)$ that the sequence $\big(\big|f(x,u_{k,m})\big|_{\frac{p}{p-1}}\big)_m$ is also bounded. It follows, using the H\"{o}lder inequality that
\begin{equation*}
\big|\int_\Omega\big(u_{k,m}-\Pi_mu_k\big)f(x,u_{k,m})\big|\leq|u_{k,m}-\Pi_mu_k|_p\big|f(x,u_{k,m})\big|_{\frac{p}{p-1}}\rightarrow0.
\end{equation*}
On the other hand, we also obtain using the H\"{o}lder inequality
\begin{equation*}
\big|\int_\Omega\phi_{u_{k,m}}u_{k,m}\big(u_{k,m}-\Pi_mu_k\big)\big|\leq|\phi_{u_{k,m}}|_3|u_{k,m}|_3|u_{k,m}-\Pi_mu_k|_3\to0.
\end{equation*}
Since $I'_{0,m}(u_{k,m})=0$, we deduce from \eqref{star} that $\|u_{k,m}\|\to\|u_k\|$, and hence that $u_{k,m}\to u_k$ in $X$, as $m\to\infty$. At this point, it is straightforward to verify that $u_k$ is a critical point of $I_0$ such that $I_0(u_k)\geq b_k$. Since $b_k\to\infty$, as $k\to\infty$ (see Remark \ref{rembk}), the proof will be completed if we show that $u_k$ is sign-changing. \\
As usual, we denote $u^\pm:=\max\{0,\pm u\}$, for any $u\in X$. Observe that
\begin{equation*}
\big<I'_{0,m}(u_{k,m}),u_{k,m}^{\pm}\big>=0\,\,\Rightarrow\,\, \|u_{k,m}^\pm\|^2\leq\int_\Omega u_{m_j}^\pm f(x,u_{k,m}^\pm).
\end{equation*}
 $(f_1)$ and $(f_2)$ imply 
\begin{equation}\label{feps}
\forall\varepsilon>0,\quad\exists c_\varepsilon>0\,\,;\,\, |f(x,t)|\leq \varepsilon|t|+c_\varepsilon|t|^{p-1},\quad\forall (x,t)\in\Omega\times\mathbb{R}.
\end{equation}
We then obtain by using the Sobolev embedding theorem
\begin{equation*}
\|u_{k,m}^\pm\|^2\leq \int_\Omega u_{k,m}^\pm f(x,u_{k,m}^\pm)\leq c\big(\varepsilon\|u_{k,m}^\pm\|^2+c_\varepsilon\|u_{k,m}^\pm\|^p\big),
\end{equation*} 
for some constant $c>0$.  Since $u_{k,m}$ is sign-changing, $u_{k,m}^\pm$ are not equal to $0$. Choosing $\varepsilon$ small enough it follows that $(\|u_{m_j}^\pm\|)$ are bounded below by strictly positive constants which do not depend on $m$. Hence $u_k$ is sign-changing.
\end{proof}
We now prove our result in the critical case.
\begin{proof}[Proof Theorem \ref{result2}]
 We suppose here that $0<\lambda<\Lambda_k^1$. We consider again the elements $u_{k,m}$ obtained in Lemma \ref{lemam}. In view of Lemma \ref{bd}, the sequence $(u_{k,m})_m$ is bounded in $X$. Hence, up to a subsequence, we have as $m\to\infty$:
\begin{align}
\nonumber u_{k,m}&\rightharpoonup u_k \text{ in }X;\\
u_{k,m}&\to u_k \text{ a.e. in }\Omega;\label{almost}\\
u_{k,m}&\to u_k \text{ in }L^r(\Omega)\,\,(1\leq r<6);\label{eq1}\\ 
u_{k,m}&\rightharpoonup u_k \text{ in }L^6(\Omega);\label{eq2}\\ 
\nonumber\phi_{u_{k,m}}&\rightharpoonup\phi_{u_k} \text{ in }X;\\
\nonumber\phi_{u_{k,m}}&\to\phi_{u_k} \text{ in }L^r(\Omega)\,\, (1\leq r<6);\\
\int_\Omega \phi_{u_{k,m}}u_{k,m}^2&\to\int_\Omega\phi_{u_k}u_k^2.\label{eq3}
\end{align} 
We claim that
\begin{equation}\label{claim}
u_{k,m}\to u_k \quad\text{if}\quad \lambda<\Lambda_k^2:=S^3/\big(3\max_{B_k}I_0\big).
\end{equation}
We fix $\lambda=\lambda_k$ such that $0<\lambda_k<\min\{\Lambda_k^1,\Lambda_k^2\}$. Using the same argument as in the proof of Theorem \ref{result1} above, we show that $u_k$ is a sign-changing critical point of $I_{\lambda_k}$ such that $I_{\lambda_k}(u_k)\geq b_k$. Since $b_k\to\infty$, as $k\to\infty$, the conclusion of Theorem \ref{result2} follows.
\par We complete the proof of Theorem \ref{result2} by proving our above claim. Let us then assume that $\lambda<\Lambda_k^2$.\\
 \eqref{eq2}, \eqref{almost} and Theorem 10.36 in \cite{W95} imply that
 \begin{equation*}
 u_{k,m}^6\rightharpoonup u_k^6 \,\,\text{ in }\,\, L^{\frac{6}{5}}(\Omega).
 \end{equation*}
For any $v\in X$ we obtain, using the H\"{o}lder inequality
\begin{align*}
\big|\int_\Omega\big(\phi_{u_{k,m}}u_{k,m}&-\phi_{u_k}u_k\big)v\big|=\big|\int_\Omega\phi_{u_{k,m}}(u_{k,m}-u_k)v\big|+\big|\int_\Omega(\phi_{u_{k,m}}-\phi_{u_k})u_kv\big|\\
&\leq |\phi_{u_{k,m}}|_3|u_{k,m}-u_k|_3|v|_3+|\phi_{u_{k,m}}-\phi_{u_k}|_3|u_k|_3|v|_3\to0,\,m\to\infty.
\end{align*}
This implies that $\phi_{u_{k,m}}u_{k,m}\rightharpoonup\phi_{u_k}u_k$ in $\mathcal{D}'(\Omega)$, the space of distributions.\\
It is clear that $-\Delta u_{k,m}\rightharpoonup -\Delta u_k$ and $f(x,u_{k,m})\rightharpoonup f(x,u_k)$ in $\mathcal{D}'(\Omega)$. Therefore, we obtain
\begin{equation*}
-\Delta u_k+\phi_{u_k}u_k=f(x,u_k)+\lambda u_k^5\,\,\text{ in }\,\, \mathcal{D}'(\Omega).
\end{equation*}
Multiplying the two members of this equation by $u_k$ and integrating, we obtain
\begin{equation}\label{onee}
\|u_k\|^2+\int_\Omega\phi_{u_k}u_k^2=\int_\Omega u_kf(x,u_k)+\lambda\int_\Omega u^6_k.
\end{equation}
On the other hand, $\big<I'_{\lambda,m}(u_{k,m}),u_{k,m}\big>=0$ is equivalent to
\begin{equation}\label{twoo}
\|u_{k,m}\|^2+\int_\Omega\phi_{u_{k,m}}u_{k,m}^2-\int_\Omega u_{k,m}f(x,u_{k,m})-\lambda |u_{k,m}|_6^6=0.
\end{equation}
It is clear that
\begin{equation}\label{u}
\|u_{k,m}\|^2=\|u_k\|^2+\|u_{k,m}-u_k\|^2+\circ(1)
\end{equation}
By Brezis-Lieb lemma \cite{BL83} we have
\begin{equation}\label{u1}
|u_{k,m}|^6_6=|u_k|_6^6+|u_{k,m}-u_k|^6_6+\circ(1).
\end{equation}
One can verify easily that \eqref{feps} implies that, for almost every $x\in\Omega$, the functions $s\mapsto sf(x,s)$ and $s\mapsto F(x,s)$ satisfy the conditions of Theorem 2 in \cite{BL83}. It then follows that
\begin{align}
\int_\Omega u_{k,m}f(x,u_{k,m})&=\int_\Omega u_{k}f(x,u_{k})+\int_\Omega \big(u_{k,m}-u_k\big)f\big(x,u_{k,m}-u_k\big)+\circ(1)\label{u2}\\
\int_\Omega F(x,u_{k,m})&=\int_\Omega F(x,u_{k})+\int_\Omega F\big(x,u_{k,m}-u_k\big)+\circ(1)\label{u3}.
\end{align}
Using \eqref{feps} and \eqref{eq1}, it is readily seen that
\begin{equation}\label{u4}
\int_\Omega \big(u_{k,m}-u_k\big)f\big(x,u_{k,m}-u_k\big)\to0
\end{equation}
\begin{equation}\label{u5}
\int_\Omega F\big(x,u_{k,m}-u_k\big)\to0.
\end{equation}
We then deduce from \eqref{twoo}, using \eqref{eq1}, \eqref{eq3}, \eqref{u}, \eqref{u1}, \eqref{u2}, and \eqref{u4} that
\begin{equation}\label{threee}
\|u_{k,m}-u_k\|^2-\lambda|u_{m,k}-u_k|_6^6=\circ(1).
\end{equation}
This implies that
\begin{equation}\label{four}
\|u_{k,m}-u_k\|^2\leq\lambda S^{-3}\|u_{m,k}-u_k\|^6+\circ(1),
\end{equation}
where $S$ is defined in \eqref{sbetak}.
If $\|u_{m,k}-u_k\|\to s_0>0$, then by \eqref{four} we would have
\begin{equation}\label{szero}
s_0\geq\big(S^3/\lambda\big)^{\frac{1}{4}}.
\end{equation}
On the other hand, in view of \eqref{eq3}, \eqref{u}, \eqref{u1}, \eqref{u3} and \eqref{u5}, we have
\begin{align*}
\max_{B_k}I_0\geq I_{\lambda,m}(u_{m,k})&=\frac{1}{2}\|u_{m,k}\|^2+\frac{1}{4}\int_\Omega\phi_{u_{m,k}}u_{m,k}^2-\int_\Omega F(x,u_{m,k})-\frac{\lambda}{6}|u_{m,k}|_6^6\\
&=\frac{1}{2}\|u_{m,k}-u_k\|^2-\frac{\lambda}{6}|u_{m,k}-u_k|_6^6+I_0(u)+\circ(1).
\end{align*}
Using \eqref{onee}, we see that
\begin{equation*}
I_0(u_k)=\Big(\frac{1}{2}-\frac{1}{4}\Big)\|u_k\|^2+\int_\Omega\Big(\frac{1}{4}u_kf(x,u_k)-F(x,u_k)\Big)+\Big(\frac{1}{4}-\frac{1}{6}\Big)|u_k|_6^6\geq0.
\end{equation*}
This implies that
\begin{align*}
\max_{B_k}I_0&\geq\frac{1}{2}\|u_{m,k}-u_k\|^2-\frac{\lambda}{6}|u_{m,k}-u_k|_6^6+\circ(1)\\
&=\frac{1}{3}\|u_{k,m}-u_k\|^2+\circ(1)\quad \big(\text{by } \eqref{threee}\big).
\end{align*}
Passing to the limit $m\to\infty$ and using \eqref{four}, we obtain 
\begin{equation*}
\lambda\geq S^3/\big(3\max_{B_k}I_0\big)=\Lambda_k^2,
\end{equation*}
which is a contradiction. Therefore \eqref{claim} holds.
\end{proof}
\section{Proof of Theorem \ref{result3}}\label{quatre}
We consider in this section the problem
\begin{equation*}
 (SP1)\qquad   \left\{
      \begin{array}{ll}
        -\Delta u+\psi u^3=|u|^{q-2}u\quad \text{in }\Omega & \hbox{} \\
        \quad\quad\,\,\,\,\,-\Delta \psi=\frac{1}{2} u^4\qquad\, \text{ in }\Omega & \hbox{} \\
       \quad\quad\,\,\, u=\psi=0 \qquad\quad\,\,\,\text{on }\partial\Omega, & \hbox{}
      \end{array}
    \right.
\end{equation*}
where $\Omega$ is a bounded domain in $\mathbb{R}^N$ ($N=1,2$) with smooth boundary and $q>8$.\\
One can verify easily that solutions of $(SP1)$ are critical points of the functional defined on $H_0^1(\Omega)\times H_0^1(\Omega)$ by
\begin{equation*}
K(u,\psi)=\frac{1}{2}\int_\Omega|\nabla u|^2-\frac{1}{4}\int_\Omega|\nabla \psi|^2+\frac{1}{4}\int_\Omega \psi u^4-\frac{1}{q}\int_\Omega|u|^q.
\end{equation*}
Since $K$ is strongly indefinite, we adopt the same strategy as in Section \ref{trois} to reduce the problem to a definite one. The following proposition can be proved in a standard way \big(see for example \cite{R06}\big).
\begin{prop}\label{propimpo}
For all $u\in H_0^1(\Omega)$, the problem
\begin{equation*}
-\Delta \psi=\frac{1}{4}u^4,\quad\psi\in H_0^1(\Omega)
\end{equation*}
possesses a unique solution $\psi_u$. Moreover, we have:
\begin{enumerate}
\item[(i)] $\psi_u\geq0$ and there exists $C>0$ such that $\|\psi_u\|\leq C\|u\|^4$;
\item[(ii)] $\psi_{tu}=t^4\psi_u$, for all $t\leq0$.
\item[(iii)] If $u_n\rightharpoonup u$ in $H_0^1(\Omega)$ then $\psi_{u_n}\rightharpoonup\psi_u$  in $H_0^1(\Omega)$ and
\begin{equation}
\int_\Omega\psi_{u_n}u_n^4\to\int_\Omega\psi_uu^4.
\end{equation}
\end{enumerate}
\end{prop}
We consider the problem
\begin{equation}
-\Delta u+\psi_uu^3=|u|^{q-2}u,\quad u\in H_0^1(\Omega)
\end{equation}
which solutions correspond to critical points of the functional
\begin{equation*}
J(u)=\frac{1}{2}\|u\|^2+\frac{1}{8}\int_\Omega \psi_uu^4-\frac{1}{q}\int_\Omega |u|^q,\quad u\in X:= H_0^1(\Omega).
\end{equation*}
By a standard argument one shows that $J$ is smooth on $H_0^1(\Omega)$ and 
\begin{equation*}
\big<J'(u),v\big>=\int_\Omega\nabla u\nabla v+\int_\Omega \psi_uu^3v-\int_\Omega|u|^{q-2}uv,\quad\forall u,v\in H_0^1(\Omega).
\end{equation*}
\begin{lem}[\cite{ADaL13}, Proposition 2.2]
Let $(u,\psi)\in H_0^1(\Omega)\times H_0^1(\Omega)$. The following statements are equivalent:
\begin{enumerate}
\item[(1)] $(u,\psi)$ is a critical point of $K$.
\item[(2)] $u$ is a critical point of $J$ and $\psi=\psi_u$.
\end{enumerate}
\end{lem}
For $k\geq2$, we consider $Y_k$ and $Z_k$ defined in \eqref{ykzk}. Similar to Lemma \ref{l1}, the following result holds.
\begin{lem}\label{lem4}\quad
\begin{enumerate}
\item[(1)] For any $u\in Y_k$, we have $J(u)\to-\infty$, as $\|u\|\to\infty$.
\item[(2)] There exists $r_k>0$ such that 
\begin{equation*}
\tilde{b_k}:=\inf_{\substack{u\in Z_k\\ \|u\|=r_k}}J(u)\to\infty,\quad\text{as }k\to\infty.
\end{equation*}
\end{enumerate}
\end{lem}
\par Now we fix $k$ large enough and we set for $m>k+2$
\begin{equation*}
J_m:=J|_{Y_m},\, K_m:=\big\{u\in Y_m\,;\, J_m'(u)=0\big\},\, E_m:=Y_m\backslash K_m,
\end{equation*}
\begin{equation*}
P_m:=\big\{u\in Y_m\,;\,u(x)\geq0\big\},\,\, Z^m_k:=\oplus_{j=k}^m X_j,\,\text{ and }  N^m_k:=\big\{u\in Z^m_k\,|\,\|u\|=r_k\big\}.
\end{equation*}
\par For any $u\in Y_m$, we denote by $v=Pu\in Y_m$ the unique solution to the problem
\begin{equation*}
-\Delta v+\psi_uu^2v=|u|^{q-2}u, \quad v\in Y_m.
\end{equation*}
Then 
\begin{equation}\label{4.3}
\int_\Omega\nabla(Pu)\nabla w+\int_\Omega\psi_{u}u^2Puw-\int_\Omega|u|^{q-2}uw=0, \quad\forall w\in Y_m.
\end{equation}
We state the analogue of Lemma \ref{Aop} for the operator $P:Y_m\to Y_m$.
\begin{lem}\label{Pop}\quad
\begin{enumerate}
\item[(1)] $P$ is continuous and it maps bounded sets to bounded sets.
\item[(2)] For any $u\in Y_m$ we have
\begin{align}
\big<J_m'(u),u-Pu\big>\geq c_1\|u-Pu\|^2,\label{4.1}\\ 
\|J_m'(u)\|\leq c_2\big(1+\|u\|^6\big)\|u-Pu\|.\label{4.2}
\end{align}
\item[(3)] There exists $\mu_m\in]0,\delta_m[$ such that $P(\pm D^0_m)\subset \pm D^0_m$, where $\delta_m$ is defined in the same manner as in \eqref{deltam}.
\end{enumerate}
\end{lem}
\begin{proof}
(1)\quad Let $u_n\to u$ in $Y_m$. By definition of $P$ we have for any $w\in Y_m$
\begin{align*}
 \int_\Omega\nabla(Pu_n)\nabla w+\int_\Omega\psi_{u_n}u_n^2Pu_nw-\int_\Omega|u_n|^{q-2}u_nw=0,\\
\int_\Omega\nabla(Pu)\nabla w+\int_\Omega\psi_{u}u^2Puw-\int_\Omega|u|^{q-2}uw=0.
\end{align*}
Taking $w=Pu_n-Pu$ in these identities and subtracting, we obtain
\begin{multline*}
\|Pu_n-Pu\|^2=\int_\Omega\big(\psi_uu^2Pu-\psi_{u_n}u_n^2Pu_n\big)(Pu_n-Pu)\\
+\int_\Omega\big(|u_n|^{q-2}u_n-|u|^{q-2}u\big)(Pu_n-Pu).
\end{multline*}
Observing that
\begin{multline*}
\psi_{u_n}u_n^2Pu_n-\psi_uu^2Pu=\psi_{u_n}(u_n-u)^2Pu_n+\big(\psi_{u_n}Pu_n-\psi_uPu\big)u^2\\
+\psi_{u_n}(u_n-u)uPu_n,
\end{multline*}
and that
\begin{align*}
\big(\psi_{u_n}Pu_n-\psi_uPu\big)(Pu-Pu_n)&=-\psi_{u_n}(Pu_n-Pu)^2+(\psi_{u_n}-\psi_u)Pu(Pu-Pu_n)\\
&\leq (\psi_{u_n}-\psi_u)Pu(Pu-Pu_n),
\end{align*}
we obtain, using the H\"{o}lder inequality and the Sobolev embedding theorem 
\begin{multline*}
\|Pu_n-Pu\|\leq C\big(\|\psi_u\|\|Pu_n\|\|u_n-u\|^2+\|\psi_{u_n}\|\|Pu_n\|\|Pu\|\|u_n-u\|\\+\|u\|^2\|Pu\||\psi_{u_n}-\psi_u|_5
+ \big||u_n|^{q-2}u_n-|u|^{q-2}u\big|_{\frac{q}{q-1}}\big).
\end{multline*}
By Proposition \ref{propimpo}
\begin{equation*}
|\psi_{u_n}-\psi_u|_5\to0.
\end{equation*}
By Theorem $A.2$ in \cite{W} 
\begin{equation*}
\big||u_n|^{q-2}u_n-|u|^{q-2}u\big|_{\frac{q}{q-1}}\to0.
\end{equation*}
It then follows that $\|Pu_n-Pu\|\to0$, that is $P$ is continuous.\\
To see that $P$ maps bounded sets to bounded sets, it suffices to take $w=Pu$ in \eqref{4.3} and use the H\"{o}lder inequality and the Sobolev embedding theorem.\\
(2)\quad Taking $w=u-Pu$ in \eqref{4.3}, we obtain
\begin{equation*}
\int_\Omega\nabla(Pu)\nabla (u-Pu)+\int_\Omega\psi_{u}u^2Pu(u-Pu)-\int_\Omega|u|^{q-2}u(u-Pu)=0.
\end{equation*}
Therefore we have
\begin{equation}
\big<J'_m(u),u-Pu\big>=\|u-Pu\|^2+\int_\Omega\psi_uu^2(u-Pu)^2\geq \|u-Pu\|^2.
\end{equation}
Using \eqref{4.3}, the H\"{o}lder inequality, the Sobolev embedding theorem, and Proposition \ref{propimpo}, we obtain for any $w$ in $Y_m$
\begin{align*}
\big|\big<J'_m(u),w\big>\big|&=\big|\int_\Omega(u-Pu)\nabla w+\int_\Omega\psi_uu^2(u-Pu)w\big|\\
&\leq \|u-Pu\|\|w\|+c_1\|u\|^6\|u-Pu\|\|w\|,
\end{align*}
which implies that
\begin{equation*}
\|J_m'(u)\|\leq c_2(1+\|u\|^6)\|u-Pu\|,
\end{equation*}
for some constants $c_1,c_2>0$.\\
(3) \quad Let $u\in Y_m$ and let $v=Pu$. Taking $w=v^+$ in \eqref{4.3}, we obtain
\begin{equation*}
\|v^+\|^2+\int_\Omega\psi_uu^2(v^+)^2=\int_\Omega v^+|u|^{q-2}u.
\end{equation*}
We then deduce, using the H\"{o}der inequality, that
\begin{equation}\label{v}
\|v^+\|^2\leq |u^+|_q^{q-1}|v^+|_q.
\end{equation}
On the other hand it is not difficult to see that $|u^+|_q\leq |u-w|_q$, for all $w\in -P_m$. Hence there is a constant $c_1=c_1(q)>0$ such that $|u^+|_q\leq c_1 dist(u,-P_m)$. It is obvious that $dist(v,-P_m)\leq \|v^+\|$. So we deduce from \eqref{v} and the Sobolev embedding theorem that
\begin{equation*}
dist(v,-P_m)\|v^+\|\leq c_2dist(u,-P_m)^{q-1}\|v^+\|.
\end{equation*}
This implies that 
\begin{equation*}
dist(v,-P_m)\leq  c_2 dist(u,-P_m)^{q-1}.
\end{equation*}
Similarly one shows that
\begin{equation*}
dist(v,P_m)\leq  c_3 dist(u,P_m)^{q-1}.
\end{equation*}
We can then find $\mu_m\in]0,\delta_m[$ small enough such that
\begin{equation*}
dist(v,\pm P_m)\leq \frac{1}{2} dist(u,\pm P_m)
\end{equation*}
whenever $dist(u,\pm P_m)<\mu_m$.
\end{proof}
\begin{lem}[\cite{LWZ}, Lemma 3.5]\label{Qop}
There exists an odd locally Lipschitz continuous operator $Q:E_m\to Y_m$ such that
\begin{enumerate}
\item[(1)] $\big<J_m'(u),u-Q(u)\big>\geq \frac{1}{2}\|u-Pu)\|^2$, for any $u\in E_m$.
\item[(2)] $\frac{1}{2}\|u-Q(u)\|\leq \|u-Pu\|\leq 2\|u-Qu\|$, for any $u\in E_m$.
\item[(3)] $Q\big((\pm D^0_m)\cap E_m\big)\subset \pm D^0_m$.
\end{enumerate} 
\end{lem}
\begin{rem}\label{rema}
Lemmas \ref{Pop} and \ref{Qop} imply that 
\begin{align*}
\big<J_m'(u),u-Q(u)\big>&\geq\frac{1}{8}\|u-Q(u)\|^2 \text{ and}\\
\|J_m'(u)\|\leq c_3(1&+\|u\|^6)\|u-Q(u)\|,\text{ for all } u\in E_m.
\end{align*}
\end{rem}
\begin{lem}\label{tech4}
Let $c<d$ and $\alpha>0$. For all $u\in Y_m$ such that $J_m(u)\in[c,d]$ and $\|J_m'(u)\|\geq\alpha$, there exists $\beta>0$ such that $\|u-B(u)\|\geq\beta$.
\end{lem}
\begin{proof}
Fix $\mu\in]8,q[$. Using \eqref{4.3}, we obtain
\begin{multline*}
J_m(u)-\frac{1}{\mu}\big<u,u-Pu\big>=\big(\frac{1}{2}-\frac{1}{\mu}\big)\|u\|^2+\big(\frac{1}{\mu}-\frac{1}{q}\big)|u|^q_q+\big(\frac{1}{8}-\frac{1}{\mu}\big)\int_\Omega\psi_uu^4\\
-\frac{1}{\mu}\int_\Omega\psi_uu^3(Pu-u).
\end{multline*}
Since the norms $\|\cdot\|$ and $|\cdot|_q$ are equivalent on $Y_m$, it follows that
\begin{equation*}
\|u\|^q+\int_\Omega\psi_uu^4\leq c\Big[|J_m(u)|+\|u\|\|u-Pu\|+\big|\int_\Omega\psi_uu^3(Pu-u)\big|\Big].
\end{equation*}
Using H\"{o}lder inequality and Sobolev embedding theorem, it is easy to see that
\begin{equation*}
\big|\int_\Omega\psi_uu^3(Pu-u)\big|\leq c_1\|u\|^3\|u-Pu\|\Big(\int_\Omega\psi_uu^4\Big)^{\frac{1}{2}}.
\end{equation*}
Noting that the function $[0,+\infty[\to\mathbb{R}$, $s\mapsto as-s^2$, attains its maximum at $a/2$, we finally obtain, using Lemma \ref{Qop}-(2)
\begin{equation}\label{4.4}
\|u\|^q\leq C\big[|J_m(u)|+\|u\|\|u-Qu\|+\|u\|^6\|u-Qu\|^2\big].
\end{equation}
Suppose that there exists a sequence $(u_n)\subset Y_m$ such that 
\begin{equation*}
J_m(u_n)\in[c,d],\quad \|J_m'(u_n)\|\geq\alpha,\quad\text{and }\|u_n-Qu_n\|\to0.
\end{equation*}
By \eqref{4.4} we see that $(\|u_n\|)$ is bounded. It follows from Remark \ref{rema} above that $J'_m(u_n)\to0$, which is a contradiction.
\end{proof}
\begin{proof}[Proof Theorem \ref{result3}]
By Lemmas \ref{lem4}, \ref{Qop} and \ref{tech4} and Remark \ref{rema}, the assumptions of Theorem \ref{scft} are satisfied. Therefore, applying Theorem \ref{scft} we obtain a sequence $(u_{k,m}^n)_n\subset V_{\frac{\mu_m}{2}}(S_m)$ such that
\begin{equation*}
\lim_{\substack{n\to\infty}}J_m'(u_{k,m}^n)=0\quad\text{and}\quad \lim_{\substack{n\to\infty}}J_m(u_{k,m}^n)\in\big[\tilde{b_k},\max_{\substack{u\in B_k}}J(u)\big],
\end{equation*}
where $\tilde{b_k}$ is defined in Lemma \ref{lem4}.\\
For any $u\in Y_m$ we have
\begin{equation*}
J_m(u)-\frac{1}{q}\big<J'_m(u),u\big>=\big(\frac{1}{2}-\frac{1}{q}\big)\|u\|^2+\big(\frac{1}{4}-\frac{1}{q}\big)\int_\Omega\psi_uu^4\geq \big(\frac{1}{2}-\frac{1}{q}\big)\|u\|^2.
\end{equation*}
This implies that
\begin{equation}\label{4.5}
\|u\|^2\leq C\big(|J_m(u)|+\|u\|\|J'_m(u)\|\big).
\end{equation}
We deduce from \eqref{4.5} that the sequence $(u_{k,m}^n)_n$ obtained above is bounded. Hence up to a subsequence $u_{k,m}^n\to u_{k,m}$ in $Y_m$, as $n\to\infty$. Since $V_{\frac{\mu_m}{2}}(S_m)$ is closed, $u_{k,m}$ belongs to $V_{\frac{\mu_m}{2}}(S_m)$. Since $J_m$ is smooth, we see that
\begin{equation*}
J_m'(u_{k,m})=0\quad\text{and}\quad J_m(u_{k,m})\in\big[\tilde{b_k},\max_{\substack{u\in B_k}}J(u)\big].
\end{equation*}
 Using the same argument as in the proof of Theorem \ref{result1}, we show that the sequence $(u_{k,m})_m$ converges, up to a subsequence, to a sign-changing critical point $u_k$ of $J$ such that $J(u_k)\geq \tilde{b_k}$. We then conclude by using the fact that $\tilde{b_k}\to\infty$, as $k\to\infty$.
\end{proof}

\section*{Acknowledgement}
This research was supported by The Fields Institute for Research in Mathematical Sciences and The Perimeter Institute for Theoretical Physics through a Fields-Perimeter Africa Postdoctoral Fellowship.

\section*{Appendix}
In this appendix, we provide the proof of Lemma \ref{defor}.
\par We first recall the following helpful lemma. 
\begin{lem}[\cite{Zou08}, Lemma 1.46]\label{helpful}
Let $\mathcal{M}$ be a closed convex subset of a Banach space $E$. If $H:\mathcal{M}\to E$ is a locally Lipschitz continuous map such that
\begin{equation*}
\lim_{\substack{\beta\to0^+}}\frac{dist\big(u+\beta H(u),\mathcal{M}\big)}{\beta}=0,\quad\forall u\in\mathcal{M},
\end{equation*}
then for any $u_0\in\mathcal{M}$, there exists $\delta>0$ such that the initial value problem
\begin{equation*}
\frac{d\sigma(t,u_0)}{dt}=H\big(\sigma(t,u_0)\big),\quad \sigma(0,u)=u_0,
\end{equation*}
has a unique solution defined on $[0,\delta)$. Moreover, $\sigma(t,u_0)\in\mathcal{M}$ for all $t\in[0,\delta)$.
\end{lem}
\begin{proof}[Proof Lemma \ref{defor}]
Define $V:E_m\to Y_m$ by $V(u)=u-B(u)$, where $B$ is given by $(A_3)$. Then there is $\delta>0$ such that $V(u)\geq\delta$ for any $u\in\Phi_m^{-1}\big([c-2\varepsilon_0,c+2\varepsilon_0]\big)\cap V_{\mu_m/2}(S_m)$ \big(in view $(A_3)$-(iii)\big). We take $\varepsilon\in ]0,\min(\varepsilon_0,\frac{\delta\alpha_1\mu_m}{8})[$ and we define
\begin{equation*}
A_1:=\Phi_m^{-1}\big([c-2\varepsilon,c+2\varepsilon]\big)\cap V_{\frac{\mu_m}{2}}(S_m),\quad A_2:=\Phi_m^{-1}\big([c-\varepsilon,c+\varepsilon]\big)\cap V_{\frac{\mu_m}{4}}(S_m),
\end{equation*}
\begin{equation*}
\chi(u):=\frac{dist(u,Y_m\backslash A_1)}{dist(u,Y_m\backslash A_1)+dist(u,A_2)}, \quad u\in Y_m
\end{equation*}
so that $\chi=0$ on $Y_m\backslash A_1$, $\chi=1$ on $A_2$, and $0\leq\chi\leq1$.\\
Consider the vector field
\begin{equation*}
W(u):=\left\{
      \begin{array}{ll}
        \chi(u)\|V(u)\|^{-2}V(u), \quad\text{ for }u\in  A_1 & \hbox{} \\
        \qquad \quad 0, \quad\textnormal{ for }u\in Y_m\backslash A_1. & \hbox{}
      \end{array}
    \right.
\end{equation*}
Clearly $W$ is odd and locally Lipschitz continuous. Moreover, by our choice of $\varepsilon$ above we have 
\begin{equation}\label{fourr}
\|W(u)\|\leq\frac{1}{\delta}\leq\frac{\alpha_1\mu_m}{8\varepsilon},\quad\forall u\in Y_m.
\end{equation}
 It follows that the Cauchy problem
\begin{equation*}
\frac{d}{dt}\sigma(t,u)=-W(\sigma(t,u)),\qquad\sigma(0,u)=u\in Y_m
\end{equation*}
has a unique solution $\sigma(\cdot, u)$ defined on $\mathbb{R}_+$. Moreover, $\sigma$ is continuous on $\mathbb{R}_+\times X$.\\
We have in view of \eqref{fourr}
\begin{equation}\label{two}
\|\sigma(t,u)-u\|\leq \int_0^t\|W(\sigma(s,u))\|ds\leq\frac{\alpha_1\mu_m}{8\varepsilon}t,
\end{equation}
and by $(A_3)$-(ii) 
\begin{align}
\nonumber \frac{d}{dt}\Phi_m(\sigma(t,u))&=-\big<\Phi_m'(\sigma(t,u)),\chi(\sigma(t,u))\|V(\sigma(t,u))\|^{-2}V(\sigma(t,u))\big>\\
& \leq -\alpha_1\chi(\sigma(t,u)).\label{three}
\end{align}
Define 
\begin{equation*}
\eta:[0,1]\times Y_m\to Y_m,\qquad \eta(t,u):=\sigma\big(\frac{2\varepsilon}{\alpha_1}t,u\big).
\end{equation*}
Conclusion (i) of the lemma is clearly satisfied and by \eqref{three} above (iii) is also satisfied. Since $W$ is odd, (v) is a consequence of the uniqueness of the solution to the above Cauchy problem.
\par We now verify (ii). Let $v\in\eta(1,\Phi_m^{-1}(]-\infty,c+\varepsilon])\cap S_m)$. Then $v=\eta(1,u)=\sigma(\frac{2\varepsilon}{\alpha_1},u)$, where $u\in \Phi_m^{-1}(]-\infty,c+\varepsilon])\cap S_m$. \\
If there exists $t\in[0,\frac{2\varepsilon}{\alpha_1}]$ such that $\Phi_m(\sigma(t,u))<c-\varepsilon$, then by (iii) we have $\Phi_m(v)<c-\varepsilon$.\\
Assume now that $\sigma(t,u)\in\Phi_m^{-1}([c-\varepsilon,c+\varepsilon])$ for all $t\in[0,\frac{2\varepsilon}{\alpha_1}]$. By \eqref{two} we have $\|\sigma(t,u)-u\|\leq\frac{\mu_m}{4}$, which means, since $u\in S_m$, that $\sigma(t,u)\in V_{\frac{\mu_m}{4}}(S_m)$. Hence $\sigma(t,u)\in A_2$ and since $\chi=1$ on $A_2$, we deduce from \eqref{three} that
\begin{equation*}
\Phi_m\big(\sigma(\frac{2\varepsilon}{\alpha_1},u)\big)\leq \Phi_m(u)-\alpha_1\int_0^{\frac{2\varepsilon}{\alpha_1}}\chi(\sigma(t,u))dt=\Phi_m(u)-2\varepsilon.
\end{equation*}
This implies, since $\Phi_m(u)\leq c+\varepsilon$, that $\Phi_m(v)=\Phi_m(\sigma(\frac{2\varepsilon}{\alpha_1},u))\leq c-\varepsilon$. Hence (ii) is satisfied.
\par It remains to verify (iv). Since $\sigma$ is odd in $u$, it suffices to show that
\begin{equation}\label{seven}
\sigma\big([0,+\infty)\times D_m^0\big)\subset D_m^0.
\end{equation}
\textbf{Claim:} We have
\begin{equation}\label{eight}
\sigma([0,+\infty)\times \overline{D_m^0})\subset \overline{D_m^0}.
\end{equation}
Assume by contradiction  that \eqref{seven} does not hold. Then there exist $u_0\in D_m^0$ and $t_0>0$ such that $\sigma(t_0,u_0)\notin D_m^0$. Choose a neighborhood $N_{u_0}$ of $u_0$ such that $N_{u_0}\subset D_m^0$. Then there is a neighborhood $N_0$ of $\sigma(t_0,u_0)$ such that $\sigma(t_0,\cdot):N_{u_0}\to N_0$ is a homeomorphism. Since $\sigma(t_0,u_0)\notin D_m^0$, the set $N_0\backslash \overline{D_m^0}$ is not empty. Hence there is $w\in N_{u_0}$ such that $\sigma(t_0,w)\in N_0\backslash \overline{D_m^0}$, contradicting \eqref{eight}.\\
We now terminate by giving the proof of our above claim.\\
By $(A_3)$-(i) we have $B( D_m^0\cap E_m)\subset D_m^0$, which implies that $B( \overline{D_m^0}\cap E_m)\subset  \overline{D_m^0}$. Obviously $\sigma(t,u)=u$ for all $t\in[0,1]$ and $u\in \overline{D^0_m}\cap K_m$.\\
Assume that $u\in \overline{D^0_m}\cap E_m$. If there is $t_1\in(0,1]$ such that $\sigma(t_1,u)\notin \overline{D_m^0}$, then there would be $s_1\in[0,t_1)$ such that $\sigma(s_1,u)\in\partial\overline{D_m^0}$ and $\sigma(t,u)\notin \overline{D_m^0}$ for all $t\in(0,t_1]$. The following Cauchy problem
\begin{equation*}
\frac{d}{dt}\mu(t,\sigma(s_1,u))=-W\big(\mu(t,\sigma(s_1,u))\big),\qquad \mu(0,\sigma(s_1,u))=\sigma(s_1,u)\in Y_m
\end{equation*}
has $\sigma(t,\sigma(s_1,u))$ as unique solution. Recalling that $W=0$ on $Y_m\backslash A_1$, we have $v-W(v)\in \overline{D_m^0}\cap (Y_m\backslash A_1)$ for any $v\in\overline{D_m^0}\cap(Y_m\backslash A_1)$.\\
Assume that $v\in A_1\cap \overline{D_m^0}$. Since $\|V(u)\|\geq\delta$, we deduce that $1-\beta\chi(v)\|V(v)\|^{-2}\geq0$ for all $\beta$ such that $0<\beta\leq \delta^2$. Recalling that $v\in\overline{D_m^0}$ implies $dist(v,P_m)\leq\mu_m$, that $V(v)=v-B(v)$, and that $aP_m+bP_m\subset P_m$ for all $a,b\geq0$ \big(because $P_m$ is a cone\big), we obtain for any $\beta\in]0,\delta^2]$
\begin{align*}
dist\big(v-\beta W(v),P_m\big)&=dist\big(v-\beta\chi(v)\|V(v)\|^{-2}V(v),P_m\big)\\
&=dist\big(v-\beta\chi\|V(v)\|^{-2}(v-B(v)),P_m\big)\\
&=dist\Big((1-\beta\chi(v)\|V(v)\|^{-2})v+\beta\chi(v)\|V(v)\|^{-2}B(v),P_m\Big)\\
&\leq dist\Big((1-\beta\chi(v)\|V(v)\|^{-2})v+\beta\chi(v)\|V(v)\|^{-2}B(v),\\
&\beta\chi(u)\|V(v)\|^{-2}P_m+(1-\beta\chi(v)\|V(v)\|^{-2})P_m\Big)\\
\leq (1&-\beta\chi(v)\|V(v)\|^{-2})dist(v,P_m)+\beta\chi(v)\|V(v)\|^{-2}dist(B(v),P_m)\\
&\leq (1-\beta\chi(v)\|V(v)\|^{-2})\mu_m+\beta\chi(v)\|V(v)\|^{-2}\mu_m\\
&=\mu_m.
\end{align*}
It follows that $v-\beta W(v)\in\overline{D_m^0}$ for $0<\beta\leq \delta^2$. This implies that
\begin{equation*}
\lim_{\substack{\beta\to0^+}}\frac{dist\big(v+\beta(- W(v)),\overline{D_m^0}\big)}{\beta}=0,\quad\forall u\in\overline{D_m^0}.
\end{equation*}
By Lemma \ref{helpful} there then exists $\delta_0>0$ such that $\sigma\big(t,\sigma(s_1,u)\big)\in \overline{D_m^0}$ for all $t\in[0,\delta_0)$. This implies that $\sigma\big(t,\sigma(s_1,u)\big)=\sigma(t+s_1,u)\in \overline{D_m^0}$ for all $t\in[0,\delta_0)$, which contradicts the definition of $s_1$. This last contradiction assures that $\sigma([0,+\infty)\times\overline{D_m^0})\subset\overline{D_m^0}$.
\end{proof}


%
%
\end{document}